\documentclass[11pt, reqno]{amsart}
\usepackage{amsmath, amsfonts, amsthm, amssymb, multicol, mathtools, dsfont, mathrsfs}
\usepackage{graphicx}
\usepackage{float, hyperref}
\usepackage{scalerel,stackengine, subcaption}
\usepackage[usenames,dvipsnames,x11names]{xcolor}
\usepackage{enumitem}
\usepackage{listings}
\usepackage{pgfplots}
\usepgfplotslibrary{fillbetween}
\pgfplotsset{width=10cm,compat=1.9}
\usepackage[square,sort,comma,numbers]{natbib}
\allowdisplaybreaks

\makeatletter
\g@addto@macro\bfseries{\boldmath}
\makeatother

\makeatletter
\def\@setauthors{%
  \begingroup
  \def\thanks{\protect\thanks@warning}%
  \trivlist
  \centering\footnotesize \@topsep30\p@\relax
  \advance\@topsep by -\baselineskip
  \item\relax
  \author@andify\authors
  \def\\{\protect\linebreak}

  \normalsize\lowercase{\authors}%
  
	\ifx\@empty\contribs
  \else
    ,\penalty-3 \space \@setcontribs
    \@closetoccontribs
  \fi
  \endtrivlist
  \endgroup
}
\def\@settitle{\begin{center}
\LARGE\lowercase{\@title}
  \end{center}%
}
\makeatother
\newcommand{\authoremail}[1]{\email{\href{mailto:#1}{\color{lightblue}{#1}}}}
\newcommand{\authoraddress}[1]{\address{\normalfont{#1}}}

\numberwithin{equation}{section}

\newtheorem{thm}{Theorem}[section]
\newtheorem{lma}[thm]{Lemma}

\newtheorem{cor}[thm]{Corollary}
\newtheorem{defn}[thm]{Definition}

\newtheorem{prop}[thm]{Proposition}

\newtheorem{conj}[thm]{Conjecture}

\renewcommand{\epsilon}{\varepsilon}
\newcommand{\eps}{\varepsilon}

\renewcommand{\geq}{\geqslant}
\renewcommand{\leq}{\leqslant}

\newcommand{\hd}{\dim_{\textup{H}}}

\makeatletter
\DeclareRobustCommand\widecheck[1]{{\mathpalette\@widecheck{#1}}}
\def\@widecheck#1#2{%
    \setbox\z@\hbox{\m@th$#1#2$}%
    \setbox\tw@\hbox{\m@th$#1%
       \widehat{%
          \vrule\@width\z@\@height\ht\z@
          \vrule\@height\z@\@width\wd\z@}$}%
    \dp\tw@-\ht\z@
    \@tempdima\ht\z@ \advance\@tempdima2\ht\tw@ \divide\@tempdima\thr@@
    \setbox\tw@\hbox{%
       \raise\@tempdima\hbox{\scalebox{1}[-1]{\lower\@tempdima\box
\tw@}}}%
    {\ooalign{\box\tw@ \cr \box\z@}}}
\makeatother

\stackMath
\newcommand\reallywidehat[1]{%
\savestack{\tmpbox}{\stretchto{%
  \scaleto{%
    \scalerel*[\widthof{\ensuremath{#1}}]{\kern.1pt\mathchar"0362\kern.1pt}%
    {\rule{0ex}{\textheight}}
  }{\textheight}%
}{2.4ex}}%
\stackon[-6.9pt]{#1}{\tmpbox}%
}
\definecolor{lightblue}{HTML}{2B77A4}
\colorlet{plotblue}{LightSkyBlue3!80}
\definecolor{darkred}{HTML}{9E0D0D}
\definecolor{purp}{HTML}{d603a9}
\definecolor{dartmouthgreen}{HTML}{00A64F}
\definecolor{Junglegreen}{HTML}{00A99A}
\definecolor{yellowcolour}{HTML}{f07c02}
\hypersetup{
	colorlinks=true,
	linkcolor=darkred,
	urlcolor=darkred,
	citecolor=lightblue
}
\title{$L^p$ averages of the discrete Fourier transform and applications}

\usepackage{fancyhdr}
\author{\large{Jonathan M. Fraser}}
\authoraddress{Jonathan M. Fraser, University of St Andrews, Scotland}
\authoremail{jmf32@st-andrews.ac.uk}
\thanks{JMF was  financially supported by an \emph{EPSRC Open Fellowship}  (EP/Z533440/1), a \emph{Leverhulme Trust Research Project Grant} (RPG-2023-281),  and an \emph{EPSRC Standard Grant} (EP/Y029550/1).}

 \author{\large{Firdavs Rakhmonov}}
\authoraddress{Firdavs Rakhmonov, University of St Andrews, Scotland}
\authoremail{fr52@st-andrews.ac.uk}
\thanks{FR was  financially supported by a  \emph{Leverhulme Trust Research Project Grant} (RPG-2023-281).}

\date{}

\begin{document}
\thispagestyle{empty}

\begin{abstract}

The discrete Fourier transform has proven to be an essential tool in many geometric and combinatorial problems in vector spaces over finite fields. In general, sets with good uniform bounds for the Fourier transform appear more `random' and are easier to analyze. However, there is a trade-off: in many cases, obtaining good uniform bounds is not possible, even in situations where many points satisfy strong pointwise bounds.  To address this limitation, the first named author proposed an approach where one attempts to replace the need for \emph{uniform} ($L^\infty$) bounds with suitable bounds for the $L^p$ \emph{average} of the Fourier transform.  In subsequent joint work, the authors applied this approach successfully to improve known results in Fourier restriction and the study of orthogonal projections. In this survey we discuss this general approach, give several examples, and exhibit some of the recent applications.

\end{abstract}
\maketitle
\tableofcontents

\section{Introduction}

Discrete Fourier analysis has long been an important tool for solving geometric and combinatorial problems in the discrete setting. Perhaps its most significant application is in the resolution of the Erdős–Turán conjecture, which asserts that any subset of the integers with positive density contains arbitrarily long arithmetic progressions. In \cite{roth}, Roth provided a partial resolution of this conjecture for progressions of length three, using an original technique based on discrete Fourier analysis.

Many problems in Euclidean harmonic analysis and geometric measure theory have also been formulated in the setting of vector spaces over finite fields; see, for example, \cite{wolff, bourgain, iosevich, dvir}. The motivation for this transition from the continuous to the finite field model is that finite fields serve as a convenient analogue of the Euclidean case, with many technical difficulties eliminated. Moreover, finite field problems are closely connected to questions in number theory and combinatorics, and techniques from these areas can often be brought to bear. However, this simplification comes with a trade-off: certain familiar tools from the Euclidean setting are no longer available. The simplest example is that finite fields lack an ordering, unlike $\mathbb{R}$. There are also numerous other quirks and subtleties that play a role.  For example, in vector spaces over finite fields, there exist non-trivial spheres of radius zero, subspaces which coincide with their orthogonal complement, and spheres which contain non-trivial affine subspaces.

In an influential paper, Iosevich and Rudnev \cite{iosevich} applied discrete Fourier analysis to a discrete analogue of the Falconer distance problem in vector spaces over finite fields. Building on this approach, the first named author \cite{fraserfinite} introduced a more nuanced framework, where one considers the $L^p$ averages of the Fourier transform instead of considering only the maximum of the Fourier transform.  This approach has a number of applications:

\begin{enumerate}
\item In \cite{fraserfinite}, various examples and applications were considered. These applications included sumset-type problems, the finite field distance problem, and the problem of counting $k$-simplices.

\medskip

\item In \cite{fraser_rakhmonov_1}, we study the problem of bounding the number of exceptional projections (those that are smaller than typical) of a subset of a vector space over a finite field onto subspaces. We establish bounds that depend on $L^p$ estimates of the Fourier transform, thereby improving various known results for sets with sufficiently good Fourier analytic properties. The special case $p=2$ recovers a recent result of Bright and Gan (following Chen), which established the finite field analogue of well-known bounds of  Peres–Schlag   from the Euclidean setting. As a further consequence, we also obtain several auxiliary results of independent interest, including a character sum identity for subspaces (solving a problem of Chen) and a generalization of Plancherel’s theorem for subspaces.

\medskip

\item In \cite{fraser_rakhmonov_2}, we address the Stein--Tomas restriction problem in the finite field setting. Mockenhaupt and Tao \cite{moctao} established a finite field analogue of the Stein--Tomas  theorem, proving that $L^r \to L^2$ restriction estimates hold for a given measure $\mu$ on a vector space over a finite field within a certain range of exponents $r$. Their result was expressed in terms of uniform bounds on the measure and its Fourier transform. We generalize their result by replacing uniform bounds on the Fourier transform with suitable $L^p$ bounds, and we show that this refinement improves the Mockenhaupt–Tao range in many cases.
\end{enumerate}

Throughout the paper, the notation $A \lesssim B$ signifies that $A \leq c B$ for some constant $c > 0$ depending only on the ambient spatial dimension $n$. Similarly, we write $A \gtrsim B$ to mean $B \lesssim A$, and $A \approx B$ if both $A \lesssim B$ and $A \gtrsim B$ hold. We will use subscripts to indicate that the implicit constants depend on other parameters, such as $p$ and $s$ in \eqref{(p,s) Salem set}. The implicit constants will never depend on the size of the base field $\mathbb{F}_q$, which is $q$. We also write $A \gg B$ to denote the negation of $A \lesssim B$.

We write $r'$ for the Hölder conjugate of $r \in [1, \infty]$, i.e., the unique $r' \in [1, \infty]$ satisfying $\frac{1}{r} + \frac{1}{r'} = 1$. Additionally, we use $|X|$ to denote the cardinality of a finite set $X$ and $\mathbb{N}_0$ to denote the set of non-negative integers.

In conclusion, we emphasize that this is a survey paper. Our aim is to explain the $L^p$ averages framework, provide some motivation for its use, and discuss several examples. We aim to make the paper accessible  to a broad mathematical audience, rather than to present full technical proofs. Accordingly, we focus on making the underlying ideas clear and accessible, while omitting certain detailed proofs, which are technical and can be found in the existing literature; see, for example, \cite{fraserfinite, fraser_rakhmonov_1, fraser_rakhmonov_2, iosevich}.

\section{Basics of discrete Fourier analysis}

As can be seen from the abstract and the introduction, we make extensive use of the Fourier transform throughout this paper, since it is the central tool around which our main results revolve. In this short section, we provide the definition of the Fourier transform in the setting of finite fields, along with some of its fundamental properties.  

Throughout the paper, we let $\mathbb{F}_q$ denote the finite field with $q$ elements, where $q$ is a power of a prime. We use $\mathbb{F}_q^{\times}$ to denote the set of nonzero elements of $\mathbb{F}_q$. Given a finite field $\mathbb{F}_q$, we may also consider $\mathbb{F}_q^n$, the $n$-dimensional vector space over $\mathbb{F}_q$, for $n \geq 1$.

By a \emph{character}, we mean any group homomorphism $\chi : (\mathbb{F}_q, +) \to (S^1, \cdot)$, where  $S^1 \coloneqq \{z \in \mathbb{C} : |z| = 1\}$.  
Note that the mapping $\chi_0 : \mathbb{F}_q \to S^1$ defined by $\chi_0(x) = 1$ for every $x \in \mathbb{F}_q$ is also a character; it is called the \emph{trivial character}.

\begin{defn}
The Fourier and inverse Fourier transforms of a function $f : \mathbb{F}_q^n \to \mathbb{C}$ are the functions $\widehat f :  \mathbb{F}_q^n \to \mathbb{C}$ and $f^{\lor} : \mathbb{F}_q^n \to \mathbb{C}$, defined by
\begin{equation*}
\widehat f(\xi) \coloneqq \sum_{x \in \mathbb{F}_q^n} f(x) \chi(-\xi \cdot x), \qquad f^{\lor}(\xi) \coloneqq \sum_{x \in \mathbb{F}_q^n} f(x) \chi(\xi \cdot x),
\end{equation*}
where $\chi$ is a nontrivial character.
\end{defn}  

The specific choice of nontrivial character $\chi$ does not play an important role in what follows and we fix one particular choice throughout. Here, $\xi \cdot x$ denotes the usual dot product in $\mathbb{F}_q^n$, which is an element of $\mathbb{F}_q$. More precisely, if $\xi = (\xi_1, \dots, \xi_n) \in \mathbb{F}_q^n$ and $x = (x_1, \dots, x_n) \in \mathbb{F}_q^n$, then  
\[
\xi \cdot x \coloneqq \xi_1 x_1 + \dots + \xi_n x_n.  
\]  

The following identity shows that one can interchange the Fourier and inverse Fourier transforms: for every $x \in \mathbb{F}_q^n$, we have
\begin{equation} 
\label{inversion}
(\widehat{f})^{\lor}(x) = \widehat{{f^{\lor}}}(x) = q^n f(x).
\end{equation}

The following result, known as \emph{Parseval's theorem}, will be very useful throughout the paper:

\begin{thm}
\label{parseval's theorem}
If $f, g : \mathbb{F}_q^n \to \mathbb{C}$, and $\widehat{f}, \widehat{g} : \mathbb{F}_q^n \to \mathbb{C}$ are their Fourier transforms, respectively, then  
\begin{equation}
\label{parseval}
\sum_{\xi \in \mathbb{F}_q^n} \widehat f(\xi) \, \overline{\widehat g(\xi)} = q^n \sum_{x \in \mathbb{F}_q^n} f(x) \, \overline{g(x)}.
\end{equation}    
\end{thm}

By taking $f = g$ in \eqref{parseval}, we obtain the following fundamental result relating the $L^2$ norms of a function $f$ and its Fourier transform $\widehat{f}$, which is known as \emph{Plancherel's theorem}.  

\begin{thm}  
If $f : \mathbb{F}_q^n \to \mathbb{C}$ and $\widehat{f} : \mathbb{F}_q^n \to \mathbb{C}$ is its Fourier transform, then  
\begin{equation}  
\label{Plancherel's thm}  
\sum_{\xi \in \mathbb{F}_q^n} |\widehat{f}(\xi)|^2 = q^n \sum_{x \in \mathbb{F}_q^n} |f(x)|^2.  
\end{equation}  
\end{thm}  

For every subset $E \subseteq \mathbb{F}_q^n$, we define $E(x)$ to be the indicator function of $E$, that is,  
\begin{equation*}  
E(x) =  
\begin{cases}  
1, & \text{if } x \in E, \\  
0, & \text{if } x \notin E.  
\end{cases}  
\end{equation*}  
In particular, by applying \eqref{Plancherel's thm} to the indicator function of $E$, we obtain the following useful identity:  
\begin{equation}  
\label{Plancherel's theorem for E}  
\sum_{\xi \in \mathbb{F}_q^n} |\widehat{E}(\xi)|^2 = q^n |E|.  
\end{equation}

Let us briefly consider what we might hope to learn from $\widehat{E}$. First, observe that
\[
|\widehat{E}(\xi)| \leq |\widehat{E}(0)| = |E|
\]
for all $\xi \in \mathbb{F}_q^n$, and, applying Plancherel's theorem \eqref{Plancherel's theorem for E}, we obtain
\[
q^n |E| = \sum_{\xi \in \mathbb{F}_q^n} |\widehat{E}(\xi)|^2 \leq |E|^2 + (q^n-1) \sup_{\xi \neq 0} |\widehat{E}(\xi)|^2.
\]
Therefore, provided $|E| \leq c q^n$ for some fixed $c \in (0,1)$, we have
\[
|E|^{\frac{1}{2}} \lesssim \sup_{\xi \neq 0} |\widehat{E}(\xi)| \leq |E|.
\]

Do we expect the largest non-zero Fourier coefficient to be close to $|E|$ or $|E|^{\frac{1}{2}}$? In fact, both are possible, and precisely where it lies in this range tells us a lot about the structure of $E$. If the largest non-zero Fourier coefficient is small (close to $|E|^{\frac{1}{2}}$), this indicates that the Fourier transform has experienced significant cancellation and therefore that $E$ is rather unstructured, i.e., almost random. On the other hand, if the largest non-zero Fourier coefficient is large (close to $|E|$), this indicates that the Fourier transform has \emph{not} experienced much cancellation (at least for some frequency $\xi \neq 0$), and there should be a good reason for this, such as $E$ being highly structured in a way that prevents cancellation.

Iosevich and Rudnev \cite{iosevich} call $E$ a \emph{Salem set} if
\[
\sup_{\xi \neq 0} |\widehat{E}(\xi)| \lesssim |E|^{\frac{1}{2}},
\]
and such sets $E$ should be thought of as being optimal from a Fourier-analytic point of view. They are `as random or unstructured as possible', and this can often be leveraged to deduce further geometric or combinatorial properties of $E$.

Our main question is as follows. Suppose $E$ is unstructured or random, but not enough to be a Salem set. What can we say about $E$? Can we use bounds such as
\[
\sup_{\xi \neq 0} |\widehat{E}(\xi)| \lesssim |E|^{\frac{3}{4}}
\]
to establish desired geometric conclusions? Or, perhaps, can we instead replace the need for uniform control of the Fourier coefficients with control of a suitable $L^p$ average? This was the novel approach introduced in \cite{fraserfinite}, and we will explore this problem in the next section and throughout the paper.

\section{\texorpdfstring{$L^p$}{Lp} averages of the Fourier transform}

In this subsection, we introduce the $L^p$ averages of the Fourier transform, following \cite{fraserfinite}, and this becomes our main object of interest. We establish the necessary framework to capture the Fourier analytic behavior of a set $E \subseteq \mathbb{F}_q^n$.

\begin{defn}
If $E \subseteq \mathbb{F}_q^n$ and $p \in [1,\infty]$, then we define the $p$-norm of its Fourier transform as follows:
\begin{equation}
\label{p-norm of FT}
\text{for } p\in [1,\infty), \quad \lVert \widehat{E}\rVert_p \coloneqq \Bigg(q^{-n} \sum_{\xi \neq 0} |\widehat{E}(\xi)|^p \Bigg)^{\frac{1}{p}},
\end{equation}
\begin{equation}
\label{inf-norm of FT}
\text{for } p = \infty, \quad \lVert \widehat{E}\rVert_{\infty} \coloneqq \sup_{\xi \neq 0} |\widehat{E}(\xi)|.
\end{equation}
\end{defn}

Notice that in the definition of the $p$-norm, we specifically exclude the origin in both \eqref{p-norm of FT} and \eqref{inf-norm of FT} to avoid certain technical issues.

\begin{defn}
For $E \subseteq \mathbb{F}_q^n$, $p \in [1,\infty]$, and $s \in [0,1]$, we say that $E$ is a $(p,s)$-Salem set if
\begin{equation}
\label{(p,s) Salem set}
\lVert \widehat{E}\rVert_p \lesssim_{p,s} |E|^{1-s}.
\end{equation}
\end{defn}

Observe that being an $(\infty, \frac{1}{2})$-Salem set is equivalent to being a Salem set in the Iosevich--Rudnev sense, and is therefore $(p, \frac{1}{2})$-Salem for every $p \in [1,\infty]$.

In general, it is of interest to determine the range of $s$ for which a given set is a $(p,s)$-Salem set. It is worth noting that the property of being a $(p,s)$-Salem set exhibits a certain concavity property, which turns out to be very useful, as reflected in the following result. 

\begin{prop}
\label{concavity}
If $E \subseteq \mathbb{F}_q^n$ is both a $(p_0,s_0)$-Salem set and a $(p_1,s_1)$-Salem set for some $1 \leq p_0 < p_1 < \infty$, then it is a $(p,s)$-Salem set for every $p \in [p_0, p_1]$, with
\begin{equation*}
s \coloneqq s_0 \frac{p_0(p_1-p)}{p(p_1-p_0)} + s_1 \frac{p_1(p-p_0)}{p(p_1-p_0)}.    
\end{equation*}
\end{prop}
\begin{proof}
See Proposition 2.1 in \cite{fraserfinite}.
\end{proof}

It is immediate that any set $E \subseteq \mathbb{F}_q^n$ is a $(p,0)$-Salem set for all $p \in [1,\infty]$. Moreover, from \eqref{Plancherel's theorem for E}, one immediately sees that every set $E \subseteq \mathbb{F}_q^n$ is a $(2,\frac{1}{2})$-Salem set. We can interpolate between this and the trivial bound at $\infty$ to obtain the following result, which also follows as a direct consequence of Proposition \ref{concavity}.

\begin{cor}
\label{p,1/p corolalry}
If $E \subseteq \mathbb{F}_q^n$ and $p \in [2,\infty]$, then $E$ is a $\big(p, \frac{1}{p}\big)$-Salem set.
\end{cor}

A natural question now is: for which sets $E$ can we beat the trivial $(p,\frac{1}{p})$ bound? We will provide plenty of examples in the next section.

\section{Examples}

In this section, we describe several examples and provide information about their Fourier transforms in the sense of $L^p$ averages.

The following two subsets of $\mathbb{F}_q^n$  are important examples: the \textit{sphere} $S^{n-1}_r$ of radius $r\in \mathbb{F}_q$ and the \textit{paraboloid} $P$, defined as follows:
\begin{equation*}
    S_r^{n-1} \coloneqq \{(x_1,\dots,x_n)\in \mathbb{F}_q^n : x_1^2+\dots+x_n^2=r\},
\end{equation*}
\begin{equation*}
    P \coloneqq \{(x_1,\dots,x_{n-1},y)\in \mathbb{F}_q^n : x_1^2+\dots+x_{n-1}^2=y\}.
\end{equation*}

The simplest geometric objects, such as lines, circles, and parabolas, which we are accustomed to seeing in the standard way in $\mathbb{R}^n$, look completely different in $\mathbb{F}_q^n$. For example, Figures~\ref{fig:line},~\ref{fig:circle}, and~\ref{fig:parabola} visualize a line, a circle, and a parabola in $\mathbb{F}_{23}^2$.

\begin{figure}[htp]
    \centering
    \includegraphics[width=10cm]{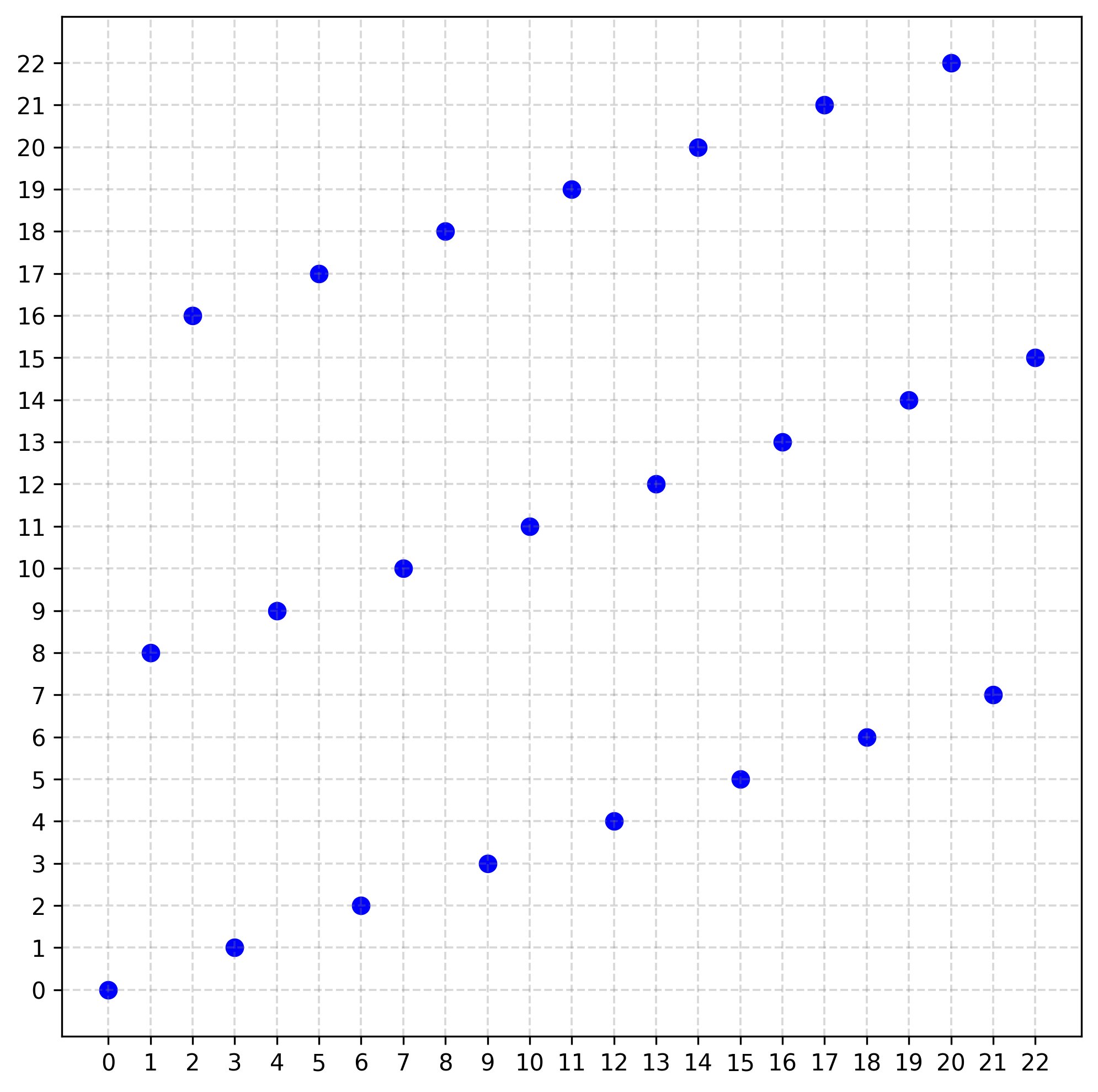}
    \caption{Line through $(0,0)$ generated by $(3,1)$ in $\mathbb{F}_{23}^2$.}
    \label{fig:line}
\end{figure}

\begin{figure}[htp]
    \centering
    \includegraphics[width=10cm]{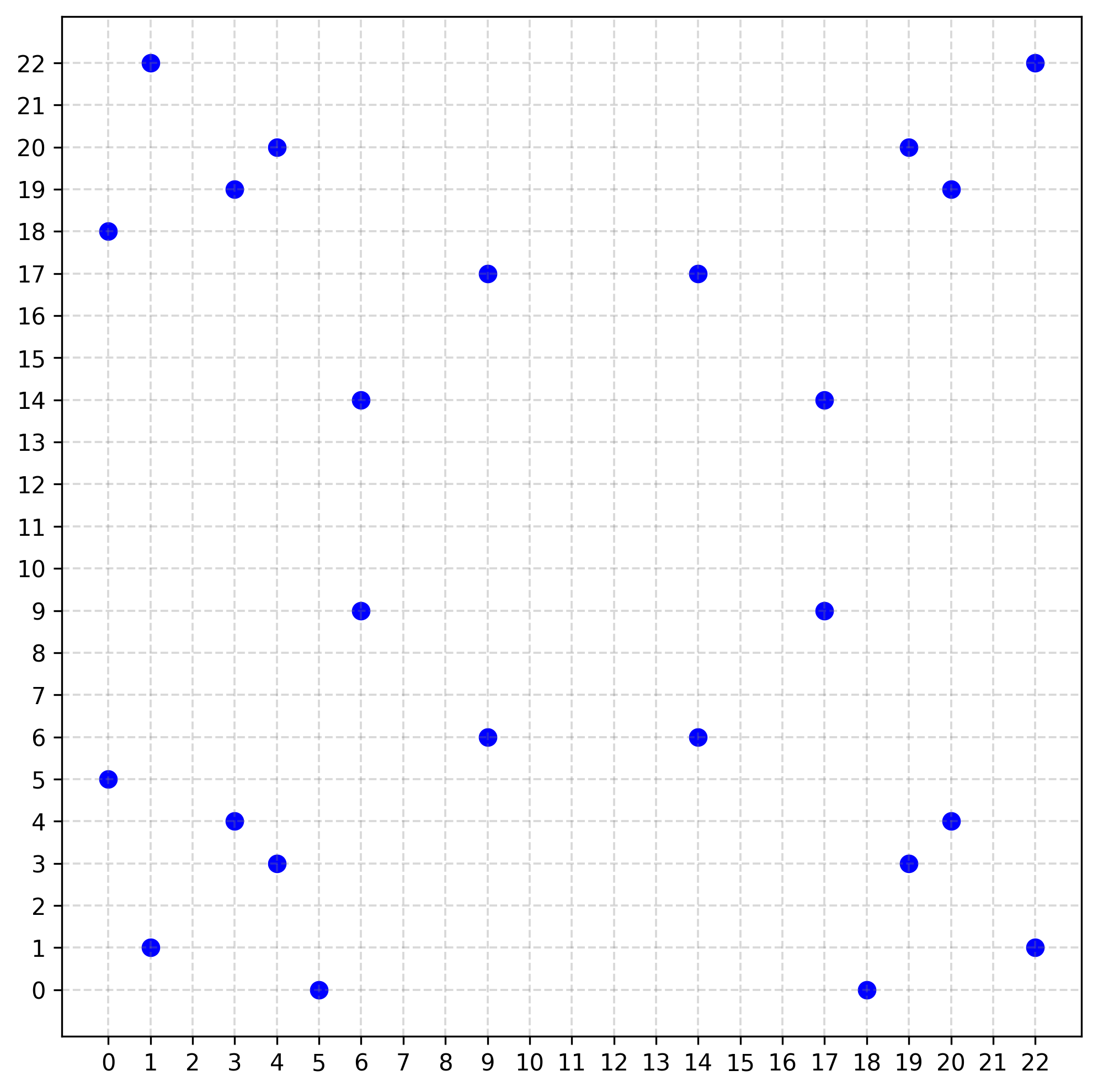}
    \caption{Circle of radius $2$ centered at $(0,0)$ in $\mathbb{F}_{23}^2$.}
    \label{fig:circle}
\end{figure}

\begin{figure}[htp]
    \centering
    \includegraphics[width=10cm]{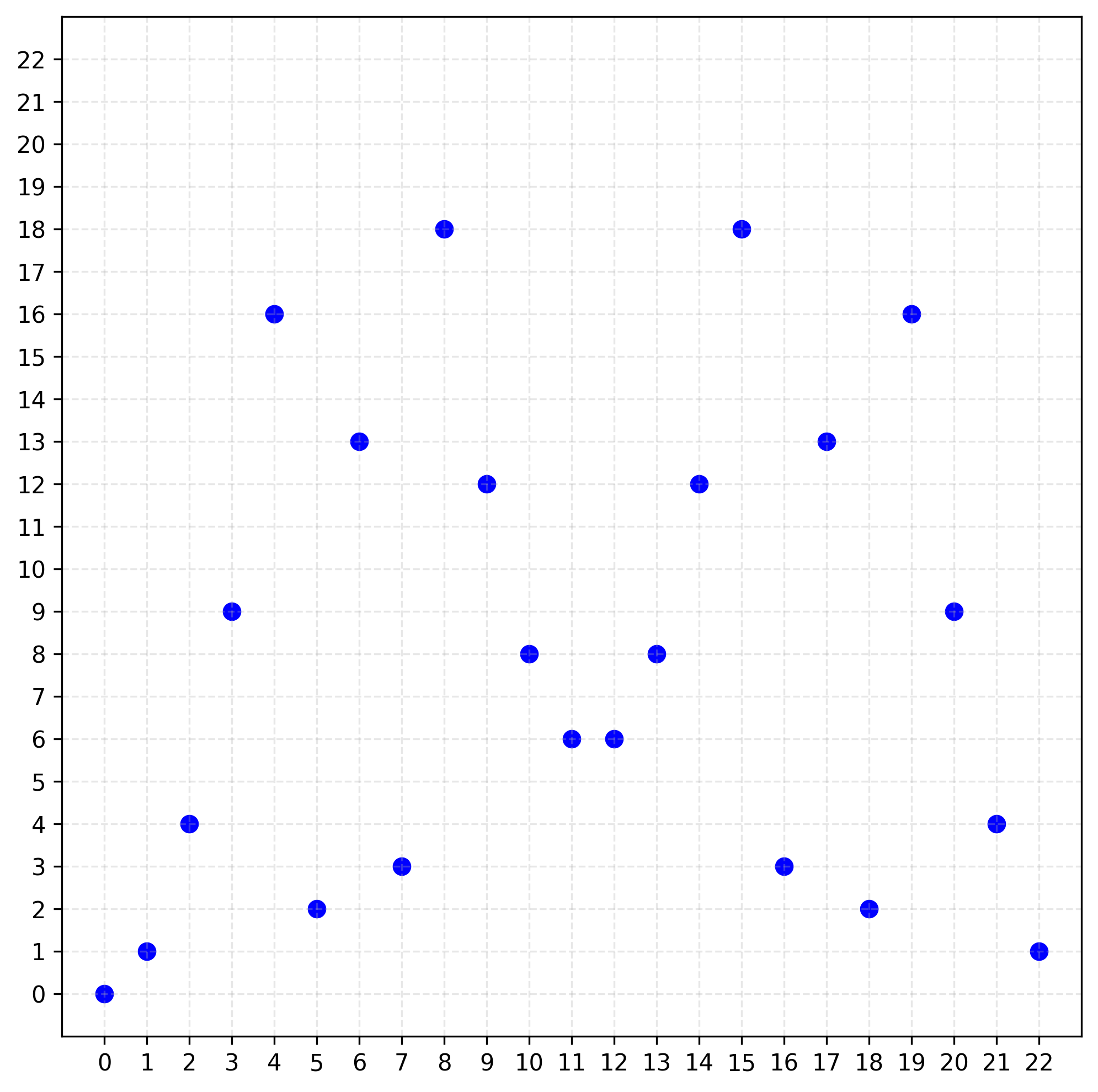}
    \caption{Parabola in $\mathbb{F}_{23}^2$.}
    \label{fig:parabola}
\end{figure}

We observe that $|S_r^{n-1}| \approx q^{n-1}$ for $r \in \mathbb{F}_q^{\times}$, as justified by the following result, which is a consequence of Theorems 6.26 and 6.27 in \cite{lidl}.

\begin{lma}
    Let $\eta$ be the quadratic character of $\mathbb{F}_q^{\times}$ with $\eta(0)=0$. Define $\lambda(r)=-1$ for $r\in \mathbb{F}_q^{\times}$ and $\lambda(0)=q-1$. Then:
    \begin{enumerate}
        \item If $n\geq 2$ is even, then
    \begin{equation}
    \label{size_of_sphere_even}
        |S_r^{n-1}| = q^{n-1} + \lambda(r)\, q^{\frac{n-2}{2}}\eta((-1)^{\frac{n}{2}}).
    \end{equation}
        \item If $n\geq 3$ is odd, then
    \begin{equation}
    \label{size_of_sphere_odd}
        |S_r^{n-1}| = q^{n-1} + q^{\frac{n-1}{2}}\eta((-1)^{\frac{n-1}{2}}r).
    \end{equation}
    \end{enumerate}
\end{lma}

The following result shows that spheres of nonzero radius and the paraboloid in $\mathbb{F}_q^n$ are Salem sets.

\begin{prop}
The sphere $S_r^{n-1}$ $(r\in \mathbb{F}_q^{\times})$ and the paraboloid $P$ are $(\infty, \tfrac{1}{2})$-Salem sets.
\end{prop}

\begin{proof}
 See Lemma 2.2 and Example 4.1 in \cite{iosevich}.
\end{proof}

Given the heuristic description of Salem sets above, a natural observation about the spheres $S_r^{n-1}$ ($r \in \mathbb{F}_q^{\times}$) is that they are neither random nor unstructured, despite being Salem. However, the real `enemy' of Fourier decay is linear structure, and one might argue that these spheres are unstructured from a linear point of view, or that they `appear' random from a Fourier-analytic perspective.

The case of the sphere of radius zero, $S_0^{n-1}$, is particularly interesting in $\mathbb{F}_q^n$ compared with $\mathbb{R}^n$. In $\mathbb{R}^n$, the zero-radius sphere is simply $S_0^{n-1} = \{0\}$. However, in $\mathbb{F}_q^n$ the situation is quite different: it may happen that $S_0^{n-1} \neq \{0\}$. For example, if $n\geq 3$ is odd and $-1$ is a square in $\mathbb{F}_q$ (i.e., $\eta(-1)=1$), then \eqref{size_of_sphere_even} gives $|S_0^{n-1}| = q^{n-1}$. On the other hand, in some cases the sphere is trivial; for instance, if $-1$ is not a square in $\mathbb{F}_q$ (i.e., $\eta(-1)=-1$), then $S_0^1\subseteq \mathbb{F}_q^2$ is trivial by \eqref{size_of_sphere_even}.

Unlike spheres of nonzero radius, $S_0^{n-1}$ is \textit{not} a Salem set. Nevertheless, one can show that for $n\geq 3$, the sphere $S_0^{n-1}$ exhibits good Fourier analytic behavior, which can be captured via the $L^p$ averages approach.

\begin{prop}
\label{sphere_zero_salem}
    Suppose that either $n\geq 3$, or $n=2$ and $-1$ is a square in $\mathbb{F}_q$. Then the sphere $S_0^{n-1}$ is a $(p,s)$-Salem set if and only if 
    \begin{equation*}
        s \leq \frac{n-2}{2(n-1)} + \frac{1}{p(n-1)}. 
    \end{equation*}
\end{prop}

\begin{proof}
See Theorem 3.1 in \cite{fraserfinite}, which relies on Proposition 3.1.6 from \cite{covert}. See also Lemma 2.2 in \cite{iosevich} for the case $p = \infty$.
\end{proof}

\begin{figure}[htp]
    \centering
    \includegraphics[width=10cm]{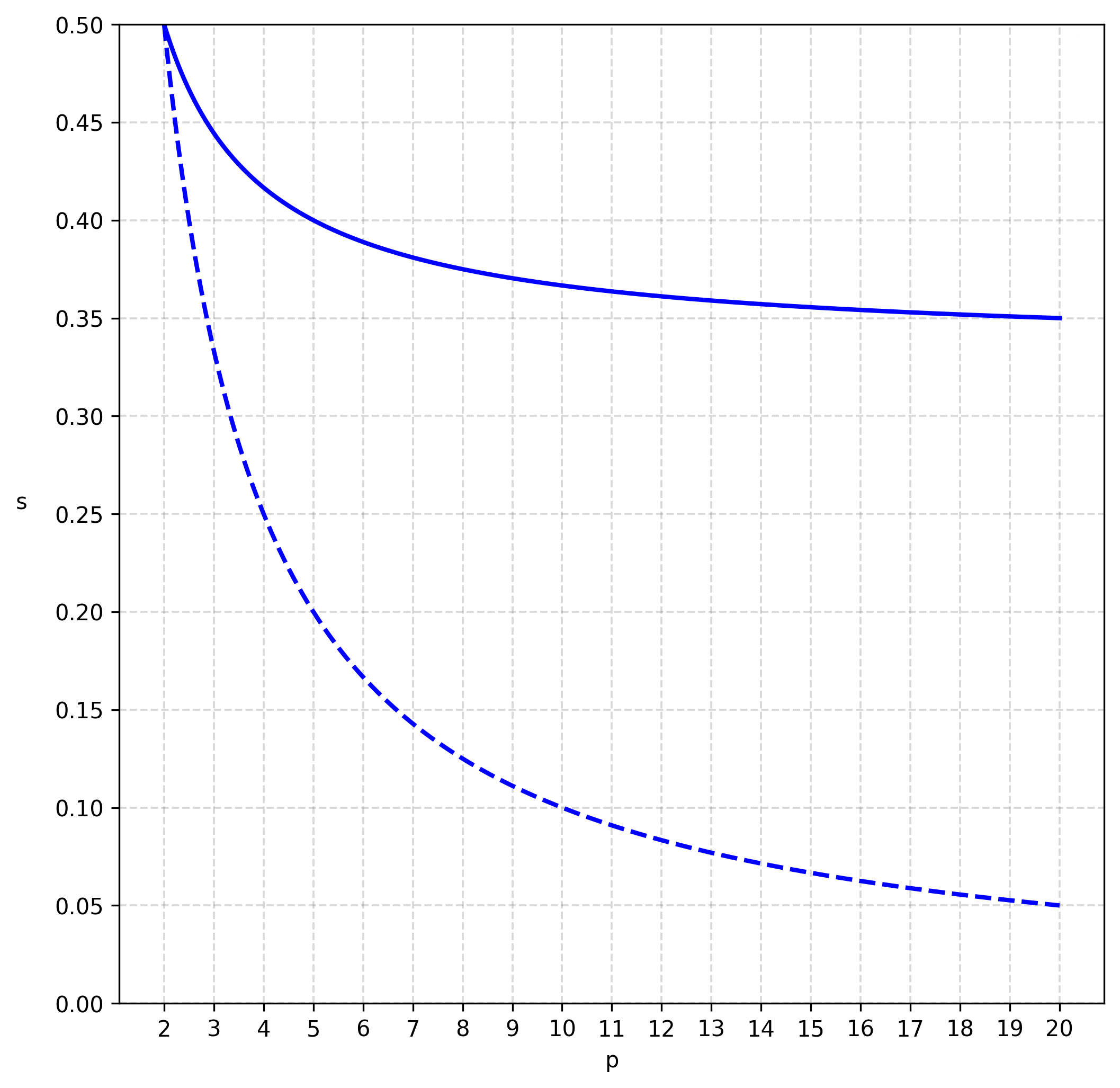}
    \caption{The threshold for which $S_0^3$ forms a $(p,s)$-Salem set is $s = \frac{1}{3} + \frac{1}{3p}$ (see Theorem \ref{sphere_zero_salem} with $n=4$), and this is plotted as a solid line. It is asymptotic to $\frac{1}{3}$ as $p \to \infty$. The trivial lower bound $(p, \frac{1}{p})$ from Corollary \ref{p,1/p corolalry} is plotted as a dashed line for comparison. One can quickly see that the sphere of radius zero exhibits good Fourier behavior on average, despite not being Salem.}
    \label{fig:circle_salem_zero}
\end{figure}

Next, consider the following sets:
\[
C^n \coloneqq \{ (z_1, \dots, z_n) \in \mathbb{F}_q^n : z_1^2+ \cdots + z_{n-2}^2 = z_{n-1}z_{n}, \, z_n \neq 0\}
\]
and
\[
D^n \coloneqq  \{ (z_1, \dots, z_n) \in \mathbb{F}_q^n : z_1^2+ \cdots + z_{n-1}^2 = z_{n}^2, \, z_n \neq 0\},
\]
both of which can be thought of as a `discrete cone'. These sets are also not Salem sets but exhibit nontrivial Fourier analytic behavior.  Perhaps curiously, the threshold is the same as for the sphere of radius zero.

\begin{prop} \label{conescor}
For $n \geq 3$, both $C^n$ and $D^n$ are $(p,s)$-Salem sets if and only if 
\[
s \leq \frac{ n-2}{2(n-1)} + \frac{1}{p(n-1)}. 
\]
In particular, neither is an $(\infty, \frac{1}{2})$-Salem set, but each is an $(\infty,s)$-Salem set if and only if $s \leq \frac{n-2}{2(n-1)}$.
\end{prop}

\begin{proof}
    See Corollary 3.11 in \cite{fraserfinite}. 
\end{proof}

Our next example is highly non-Euclidean and is more closely related to the geometry of finite fields. While subspaces themselves exhibit trivial Fourier behavior, their complements display much more interesting behavior.

\begin{prop} \label{smallcomplement}
Let $1\leq k <n$, and define $E_k\coloneqq \mathbb{F}_{q}^k \times\{0\} \subseteq \mathbb{F}_{q}^n$, and $E \coloneqq \mathbb{F}_{q}^n \setminus E_k$.  Then $E$ is a $(p,s)$-Salem set if and only if 
\[
s \leq 1 -\frac{k}{n}+ \frac{k}{pn}.
\]
In particular,  $E$ is an $(\infty, \frac{1}{2})$-Salem set if and only if $k \leq \frac{n}{2}$.
\end{prop}

\begin{proof}
    See Proposition 3.8 in \cite{fraserfinite}. 
\end{proof}

The above demonstrates that a set can be a $(4, \frac{1}{2})$-Salem set without being an $(\infty, \frac{1}{2})$-Salem set. Indeed, this occurs whenever $\frac{n}{2} < k \leq \frac{2n}{3}$.

Next, we consider certain algebraic sets, i.e., sets defined by polynomials. First, we examine an example of a flat. It was observed in \cite[Example 4.2]{iosevich} that the set $\{(k, k) : k \in \mathbb{F}_q\} \subseteq \mathbb{F}_q^2$ is not a Salem set. The following more general result can be found in \cite[Corollary 3.13]{fraserfinite}, where it is deduced as a special case of \cite[Proposition 3.12]{fraserfinite}. Here, we give a simple direct proof. Note that this result shows that flats are `as bad as possible' from a Fourier-analytic point of view, at least in the context of the $L^p$ averages framework.

\begin{prop}
 Let
\[
E \coloneqq \{ (k, \dots, k) : k \in \mathbb{F}_{q}\} \subseteq \mathbb{F}_{q}^n.
\]
Then $E$ is a $(p, s)$-Salem set if and only if $s \leq \frac{1}{p}$.
\end{prop}
\begin{proof}
By direct calculation, for $p \geq 1$, we have:
\begin{align*}
    \lVert \widehat{E}\rVert_p &= \Bigg(q^{-n} \sum_{\xi \neq 0} |\widehat{E}(\xi)|^p \Bigg)^{\frac{1}{p}} \\
     &= \Bigg(q^{-n} \sum_{\xi = (\xi_1, \dots, \xi_n) \neq 0} \Bigg\lvert  \sum_{k \in \mathbb{F}_q} \chi\big(-k(\xi_1 + \cdots + \xi_n)\big)\Bigg\rvert^p \Bigg)^{\frac{1}{p}} \\
      &= \Bigg(q^{-n} \sum_{\substack{\xi \neq 0:\\
      \xi_1 + \cdots + \xi_n =0}} q^p \Bigg)^{\frac{1}{p}} \\
       &\approx \Bigg(q^{-n} q^{n-1}  q^p \Bigg)^{\frac{1}{p}} \\
       &= q^{1-\frac{1}{p}},
\end{align*}
completing the proof. Here, we used the simple but fundamental  facts that
\[
\sum_{x \in \mathbb{F}_q} \chi(x) = 0, \qquad \sum_{x \in \mathbb{F}_q} \chi(0) = q,
\]
which are often central to this type of calculation.
\end{proof}

\begin{figure}[htp]
    \centering
    \includegraphics[width=10cm]{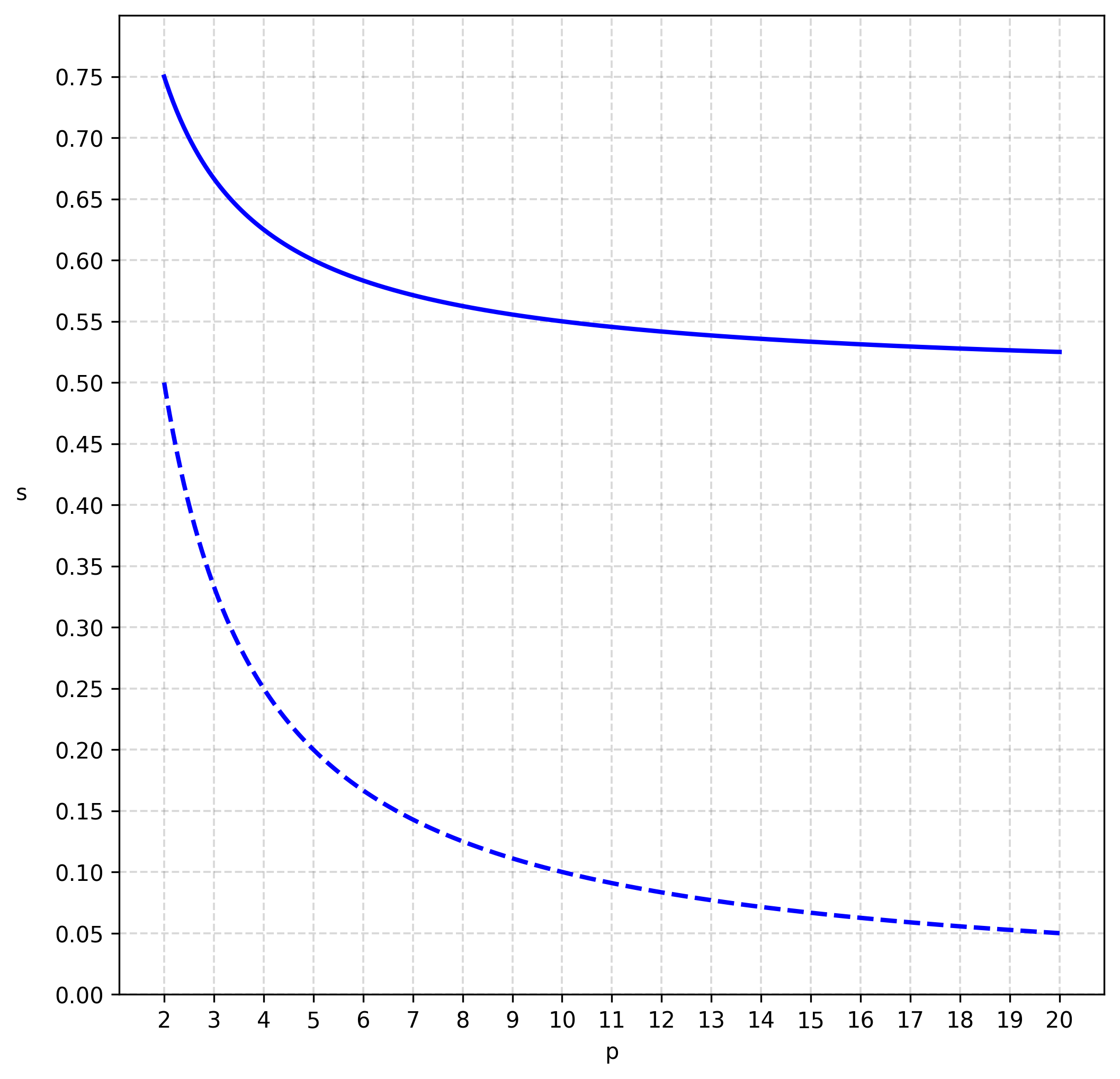}
    \caption{The threshold at which $\mathbb{F}_q^2 \setminus (\mathbb{F}_q \times \{0\})$ forms a $(p,s)$-Salem set is $s = \frac{1}{2} + \frac{1}{2p}$ (see Theorem \ref{smallcomplement} with $n=2$ and $k=1$), and this is plotted as a solid line. It is asymptotic to $\frac{1}{2}$ as $p \to \infty$. The trivial lower bound $(p, \frac{1}{p})$ from Corollary \ref{p,1/p corolalry} is plotted as a dashed line for comparison.}
    \label{fig:subspace}
\end{figure}

To obtain nontrivial Fourier behavior, we need to add some `curvature'.

\begin{prop}
For $n \geq 2$, let
\[
E \coloneqq \{ (k, \dots, k, k^{-1}) : k \in \mathbb{F}_{q}^{\times}\} \subseteq \mathbb{F}_{q}^n.
\]
If $n=2$, then $E$ is an $(\infty, \frac{1}{2})$-Salem set. On the other hand, if $n \geq 3$,  $E$ is a $(p,\frac{2}{p})$-Salem set for all $p\geq 4$.
\end{prop}

\begin{proof}
    See Proposition 3.14 in \cite{fraserfinite}. Unsurprisingly, this result is proved by appealing to Kloosterman sums. 
\end{proof}

By replacing the Kloosterman sums in the previous result with more general character sums, one obtains a new general class of Salem sets.

\begin{thm} \label{curves}
For $n \geq 2$, let
\[
E \coloneqq \{ (f_1(k), \dots, f_n(k)) : k \in \mathbb{F}_{q}\} \subseteq \mathbb{F}_{q}^n,
\]
where $f_1, \dots, f_n \in \mathbb{F}_{q}[x]$.  Suppose $f_1, \dots, f_n$  span an $m$-dimensional subspace of $\mathbb{F}_{q}[x]$.  If $n>m$,  then $E$ is $(p,\frac{m}{p})$-Salem set for all $p \geq 2m$.  On the other hand, if  $n=m$, that is, if $f_1, \dots, f_n$ are linearly independent polynomials, then $E$ is an $(\infty,\frac{1}{2})$-Salem set.
\end{thm}

\begin{proof}
    See Proposition 3.15 in \cite{fraserfinite}.
\end{proof}

An especially simple example covered by Theorem \ref{curves} is the \emph{Veronese curve}.
\begin{cor}
The rational normal curve (or Veronese curve)
\[
\{(k, k^2, \dots , k^n) : k \in \mathbb{F}_q\} \subseteq \mathbb{F}_q^n
\]
is an $(\infty, \frac{1}{2})$-Salem set in $ \mathbb{F}_q^n$.
\end{cor}

\begin{figure}[htp]
    \centering
    \includegraphics[width=10cm]{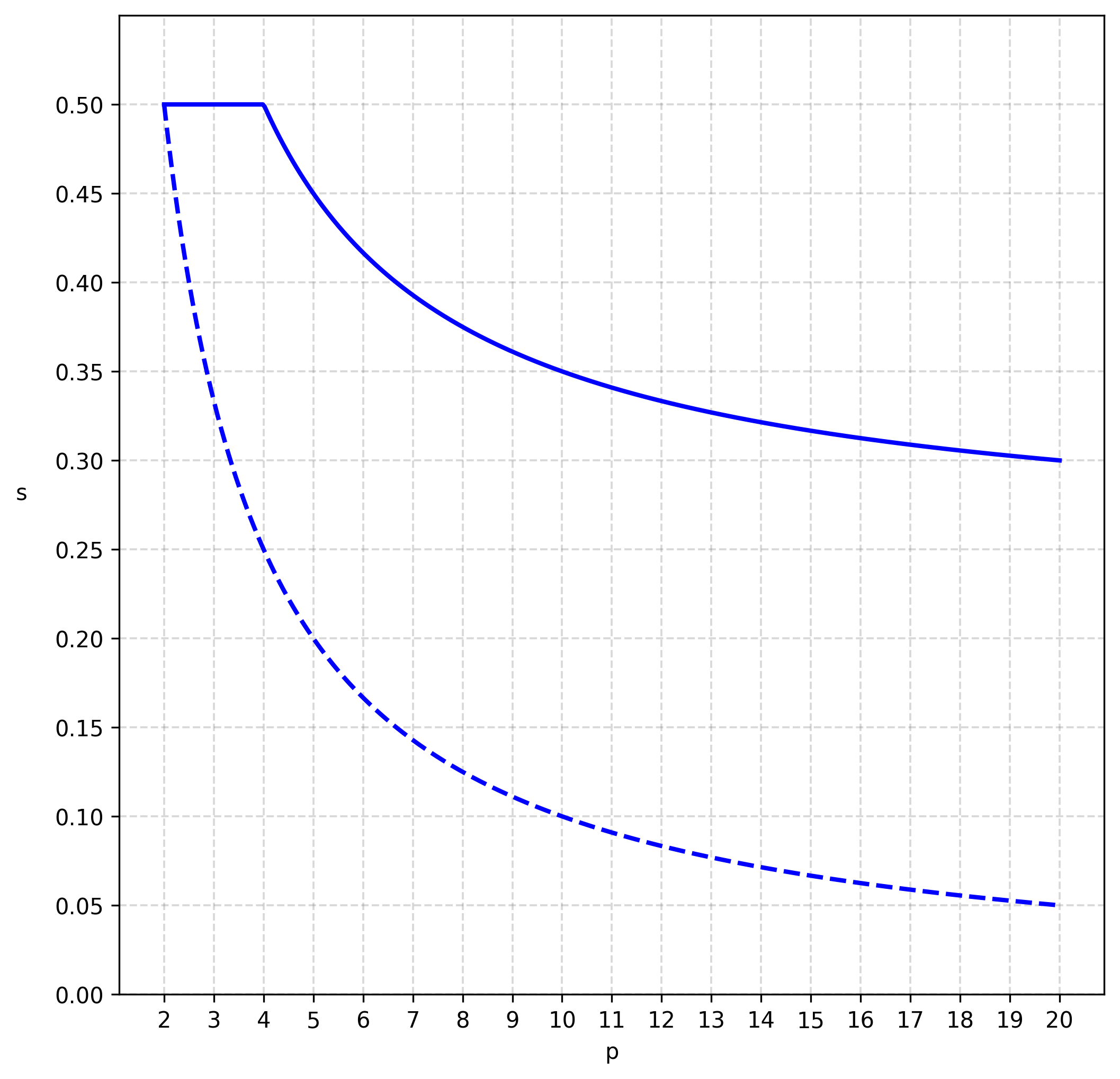}
    \caption{The threshold at which the Hamming variety in $\mathbb{F}_q^5$ forms a $(p,s)$-Salem set is $s = \min\{\frac{1}{2}, \frac{1}{4} + \frac{1}{p}\}$ (see Proposition \ref{hammingformula} with $n=5$), and this is plotted as a solid line. It is asymptotic to $\frac{1}{4}$ as $p \to \infty$. The trivial lower bound $(p, \frac{1}{p})$ from Corollary \ref{p,1/p corolalry} is plotted as a dashed line for comparison.}
    \label{fig:hamming}
\end{figure}

We provide one more example that will be needed later. For each $j\in \mathbb{F}_q^{\times}$, the Hamming variety $H_j$ in $\mathbb{F}_q^n$ is defined as
\begin{equation*}
    H_j \coloneqq\Big\{(x_1,\dots,x_n)\in \mathbb{F}_q^n: \prod_{k=1}^nx_k=j\Big\}.
\end{equation*}

Since $j\in \mathbb{F}_q^{\times}$, it is straightforward to verify that $|H_j|=(q-1)^{n-1}\approx q^{n-1}$.

\begin{prop} \label{hammingformula}
    The Hamming variety $H_j$ is a $(p,s)$-Salem set if and only if 
    \[
    s \leq \min\left\{\frac{1}{2}, \frac{1}{n-1}+\frac{1}{p} \right\}.
    \]
\end{prop}
\begin{proof}
    See Proposition 5.4 in \cite{fraser_rakhmonov_2}, which also relies on estimates for the Fourier transform given in \cite{hamming}.
\end{proof}

\section{Applications}

\subsection{Sumsets}

Many problems in additive combinatorics and additive number theory revolve around the study of sumsets of specific sets $A,B$. For example, if  
\[
\mathbb{N}_0^2 \coloneqq \{0,1,4,9,16,\dots\}
\]  
is the set of square integers, then the famous theorem of Lagrange states that $\mathbb{N} = 4\mathbb{N}_0^2$, i.e., every natural number can be expressed as the sum of four squares. Some interesting estimates for, and problems concerning,  sumsets can be found in \cite{taovu}. 

Given non-empty sets $A,B\subseteq \mathbb{F}_q^n$, the \textit{sumset} is defined by  
\begin{equation*}
    A+B \coloneqq \{a+b: a\in A,\, b\in B\}\subseteq \mathbb{F}_q^n.
\end{equation*}

A key problem is to relate $|A+B|$ to $|A|$ and $|B|$. Clearly, one has the following trivial lower and upper bounds:
\begin{equation}
\label{trivial_lower_upper_bound_sumset}
\max\{|A|,|B|\} \leq |A+B| \leq \min\{q^n,\, |A||B|\},
\end{equation}
and these bounds cannot be improved in general. Of particular interest is to determine under what conditions the bounds in \eqref{trivial_lower_upper_bound_sumset} can be sharpened; for example, establishing  growth of the form  
\begin{equation*}
    |A+A|\gtrsim |A|^{1+\varepsilon}
\end{equation*}
for some $\varepsilon>0$.

In the next result, we show that such an improvement can be obtained using the $L^p$ averages approach. In particular, we focus on the $L^4$ averages and note that the $L^2$ averages, for example,   cannot yield such a result. Indeed, if $A$ is an arithmetic progression, then $|A + A| \approx |A|$, but $A$ is a $(2, \tfrac{1}{2})$-Salem set (that is, $A$ has optimal Fourier analytic behaviour in an $L^2$ sense).  We give the proof in this case as a simple example exhibiting how the $L^4$ average can be used.  A more general result concerning $k$-fold sumsets of distinct sets and general $L^p$ averages can be found in \cite[Theorem 6.1]{fraserfinite}; see also \cite[Lemma 3.1]{imp}.

\begin{thm} \label{sumsets}
Let $A\subseteq \mathbb{F}_q^n$ be a $(4,s)$-Salem set. Then
\begin{equation*}
    |A+A| \gtrsim \min\left\{q^{n},\, |A|^{4s}\right\}.
\end{equation*}
In particular, if $s=\frac{1}{2}$, then 
\begin{equation*}
    |A+A| \approx \min\left\{q^{n},\, |A|^{2}\right\},
\end{equation*}
and we obtain the optimal additive growth. Moreover, as long as $s>\frac{1}{4}$, we obtain some improvement on the trivial lower bound. 
\end{thm}

\begin{proof}
Define $f: \mathbb{F}_q^n \to \mathbb{R}$ by
\[
f(z) \coloneqq \sum_{\substack{x,y \in \mathbb{F}_q^n\\x+y=z}} A(x)A(y).
\]
It is straightforward to check that
\begin{equation} \label{summing}
\sum_{z \in \mathbb{F}_q^n} f(z) = |A|^2,
\end{equation}
and
\begin{equation}
\label{support of f}
A+A = \{z\in \mathbb{F}_q^n : f(z) \neq 0\}.
\end{equation}
By the definition of the Fourier transform, we obtain
\begin{equation}
\label{fourrr}
\widehat{ f} (\xi) = \sum_{x,y\in\mathbb{F}_q^n} \chi(-\xi \cdot (x+y)) A(x)A(y) = \widehat{A}(\xi)\widehat{A}(\xi).
\end{equation} 
Therefore,
\begin{align*}
|A|^4 &= \Bigg( \sum_{z \in \mathbb{F}_q^n} f(z) \Bigg)^2 \qquad \text{(by \eqref{summing})} \\
&\leq |A+A| \sum_{z \in \mathbb{F}_q^n}  f(z)^2 \qquad \text{(by Jensen's inequality and \eqref{support of f})} \\
&= q^{-n}|A+A| \sum_{\xi \in \mathbb{F}_q^n}  |\widehat{f} (\xi)|^2 \qquad \text{(by Plancherel’s theorem \eqref{Plancherel's thm})} \\
&= q^{-n}|A+A| \sum_{\xi \in \mathbb{F}_q^n}  |\widehat{A}(\xi)|^4   \qquad \text{(by \eqref{fourrr})} \\
&= q^{-n} |A+A|  |\widehat{A}(0)|^4   + q^{-n}|A+A|\sum_{\xi\neq 0}  |\widehat{A}(\xi)|^4 \\
&\lesssim q^{-n} |A|^4|A+A| + \|\widehat{A}\|_{4}^4 |A+A| \\
&\lesssim q^{-n}|A|^4|A+A| + |A|^{4(1-s)} |A+A|.
\end{align*}
Hence, we conclude that
\[
|A+A|\gtrsim \min\left\{q^{n},\, |A|^{4s} \right\},
\]
as required.
\end{proof}

\subsection{Distance sets} 

The distinct distances problem was introduced by Erdős \cite{erdos}. For a finite set $E\subseteq \mathbb{R}^2$, let $\Delta(E)$ denote the set of distances spanned by pairs of points of $E$, that is,
\begin{equation*}
    \Delta(E)\coloneqq\{|x-y|: x,y\in E\}.
\end{equation*}
Each distance appears in $\Delta(E)$ at most once, regardless of how many pairs of points span it; this is why we refer to $\Delta(E)$ as the set of distinct distances of $E$. The distinct distances problem asks for
\[
\min \{|\Delta(E)|: E\subseteq \mathbb{R}^2, |E|=n\}.
\]
In other words, what is the minimum number of distinct distances determined by a set of $n$ points in $\mathbb{R}^2$? 

The continuous analogue of Erd\H{o}s' distinct distances problem is called Falconer's distance problem. It asks for the smallest Hausdorff dimension of a subset $E \subseteq \mathbb{R}^d$ $(d\geq 2)$ such that the Lebesgue measure of the distance set
\[
\Delta(E)=\{|x-y|: x,y\in E\}
\]
is positive. 

One can consider the Falconer problem in vector spaces over finite fields as a discrete model of the continuous version. We define the function $\lVert \cdot \rVert: \mathbb{F}_q^n \to \mathbb{F}_q$ by  
\begin{equation*}
\lVert x \rVert \coloneqq x_1^2 + \dots + x_n^2    
\end{equation*}
for $x=(x_1,\dots,x_n) \in \mathbb{F}_q^n$. It is worth noting that this function is not a norm, and we do not impose any metric structure on $\mathbb{F}_q^n$. Nevertheless, it shares an important feature with the Euclidean norm: invariance under orthogonal transformations. 

Given $E\subseteq \mathbb F_q^n$, the \emph{distance set} of $E$ is defined as  
\begin{equation*}
    \Delta(E)\coloneqq \{ \lVert x-y\rVert: x,y\in E\} \subseteq \mathbb{F}_q.    
\end{equation*}

A well-known and notoriously difficult problem is to obtain a sharp lower bound for $|\Delta(E)|$ in terms of $|E|$. This problem was proposed by Iosevich and Rudnev in \cite{iosevich} as a finite field analogue of Falconer's problem in Euclidean space, see below. It is also closely related to the Erdős distinct distances problem over finite fields, introduced by Bourgain, Katz, and Tao in \cite{bourgain}. Consequently, this problem is often referred to as the Erdős--Falconer distance problem.

In the finite field setting, the Falconer distance problem can be formulated as follows: find the smallest exponent $\alpha > 0$ such that, for any $E \subseteq \mathbb{F}_q^n$ with $|E| \geq Cq^\alpha$, we have $|\Delta(E)| \geq cq$, where $C > 1$ is a sufficiently large constant and $0 < c \leq 1$ is a constant independent of both $q$ and $|E|$. 

\begin{conj}
\label{maindistanceconj}
Let $q$ be odd and $n$ even. If $E\subseteq \mathbb{F}_q^n$ and $|E|\geq Cq^{\frac{n}{2}}$ with $C$ sufficiently large, then $|\Delta(E)|\gtrsim q$.
\end{conj}

Iosevich--Rudnev \cite{iosevich} also made some progress towards the above conjecture by establishing the following result. 

\begin{thm}[Iosevich--Rudnev]
\label{IosvichRudnev}
Let $E\subseteq \mathbb{F}_q^n$ with $n\geq 2$. If $|E|>4q^{\frac{n+1}{2}}$, then $\Delta(E)=\mathbb{F}_q$.
\end{thm}

We note that in \cite{firdavs}, Theorem \ref{IosvichRudnev} was generalized to arbitrary non-degenerate quadratic forms. Moreover, to better understand the Erdős–Falconer distance problem, several generalized and modified versions of this problem have been introduced and studied; see, for example, \cite{kang_koh_rakhmonov}.

The assumption that $n$ is even in Conjecture \ref{maindistanceconj} is necessary. Indeed, it was shown in \cite{hart} that the conjecture fails for odd $n$, and that in this case the correct threshold is indeed $\tfrac{n+1}{2}$. The assumption that $q$ is odd is also necessary. For example, if $q = 2^m$ with large $m$ and $E \coloneqq S_0^{n-1}$ denotes the sphere of radius zero, then $|E| \approx q^{n-1}$ but $\Delta(E) = \{0\}$. This follows immediately from the fact that, in characteristic $2$, we have $\lVert x - y \rVert = \lVert x \rVert + \lVert y \rVert$.

Iosevich and Rudnev introduced a Fourier analytic approach to the finite field distance conjecture by studying discrete analogues of Mattila integrals. In particular, they proved that if $E \subseteq \mathbb{F}_q^n$ satisfies $|E| \geq C q^{\frac{n}{2}}$ with $C$ sufficiently large and $E$ is an $(\infty, \tfrac{1}{2})$-Salem set, then $|\Delta(E)| \gtrsim q$. Using the $L^p$ averaging approach, we can strengthen this result, obtaining in particular a solution for $(4,\tfrac{1}{2})$-Salem sets, which form a significantly larger family than $(\infty,\tfrac{1}{2})$-Salem sets.

\begin{thm} \label{distancemain}
Let $q$ be odd and suppose $E \subseteq \mathbb{F}_q^n$ satisfies $|E| \geq C q^{\frac{n}{2}}$ with $C$ sufficiently large. If $E$ is a $(4,s)$-Salem set, then
\[
|\Delta(E)| \gtrsim \min\left\{q, q^{1-n} |E|^{4s} \right\}.
\]
In particular, if $E$ is a $(4,\tfrac{1}{2})$-Salem set, i.e.,
\[
 \sum_{\xi  \in \mathbb{F}_q^n \setminus\{0\}} |\widehat E(\xi)|^{4}   \lesssim q^{n}|E|^2,
\]
then $|\Delta(E)|\gtrsim q.$
\end{thm}

\begin{proof}
See Theorem 9.3 in \cite{fraserfinite} and also estimates  from \cite{iosevich}.
\end{proof}

\subsection{Exceptional projections}

Marstrand's projection theorem is one of the most fundamental results in fractal geometry, see \cite{projectionss} for more background on the theorem and its many variants. It states that for a Borel set $E \subseteq \mathbb{R}^n$ with Hausdorff dimension $\hd E$, the Hausdorff dimension of the orthogonal projection of $E$ onto almost all $k$-dimensional subspaces is $\min\{\hd E,k\}$. Due to the work of Mattila, Falconer, Bourgain, Peres--Schlag, and others, we know the following refinement of Marstrand's theorem, stated in a form due to Mattila \cite{mattila75} and Peres--Schlag \cite{peres}. For a Borel set $E \subseteq \mathbb{R}^n$,
\begin{equation} 
\label{marstrand}
\hd \{ V \in G(k,n) : \hd \pi_V(E) \leq u \} \leq k(n-k) + u-\max\{\hd E ,k\}
\end{equation}
for all $0 \leq u < \min\{\hd E,k\}$ such that the right-hand side is non-negative. Here, $\pi_V(E)$ denotes the orthogonal projection of $E$ onto $V$, and $G(k,n)$ is the Grassmannian
manifold consisting of all $k$-dimensional linear subspaces of $\mathbb{R}^n$.

One can consider Marstrand's projection theorem in the setting of finite fields, and the appropriate analogue of \eqref{marstrand} is the following: For $E\subseteq \mathbb{F}_q^n$, 
\begin{equation} 
\label{marstrandfinite}
    |\{V\in G(k,n): |\pi_V(E)|\leq u\}|\lesssim  \frac{q^{k(n-k)} u }{\max\{ |E|, q^{k}\}}
\end{equation}
for all $0 \leq u \leq q^{-\eps} \min\{ |E|, q^k\}$ for some $\eps>0$. Here, $G(k, n)$ again denotes the set of all $k$-dimensional linear subspaces of $\mathbb{F}_q^n$, and $\pi_V(E)$ is the projection of $E \subseteq \mathbb{F}_q^n$ onto the subspace $V$ of $\mathbb{F}_q^n$ (for the precise definition of projection in $\mathbb{F}_q^n$, see Definition \ref{projdef}).

We emphasize that in the finite field setting, $G(k,n)$ is directly related to the combinatorial object known as the \emph{Gaussian binomial coefficient} or \emph{$q$-binomial coefficient}, which we define below.

\begin{defn}
\label{Gaussian binomial}
    Let $k,n\in \mathbb{N}_0$, and let $q$ be a power of a prime. The Gaussian binomial coefficient, or $q$-binomial coefficient, is defined as
    \begin{equation*}
            \binom{n}{k}_q\coloneqq\begin{cases}
\dfrac{(q^n-1)(q^n-q)\dots(q^n-q^{k-1})}{(q^k-1)(q^k-q)\dots(q^k-q^{k-1})}, & \text{if } k\leq n,\\
0, & \text{if } k>n. 
\end{cases}
    \end{equation*}
    Note that $\binom{n}{0}_q=1$ because both the numerator and the denominator are empty products. 
\end{defn}

The following lemma demonstrates the relationship between $G(k, n)$ and $\binom{n}{k}_q$.

\begin{lma}
\label{relation between G(k,n) and GBC}
Let $k,n\in \mathbb{N}_0$. If $0\leq k\leq n$, then
\begin{equation}
\label{size of G(k,n)}
|G(k,n)|=\binom{n}{k}_q.    
\end{equation}
\end{lma}

There are many other identities connecting $G(k,n)$ and $\binom{n}{k}_q$, which are used extensively to prove Marstrand's projection theorem in finite fields (Theorem \ref{main projection theorem}); for example, see Lemma 2.3 in \cite{fraser_rakhmonov_1}. Next, we formally define a \emph{projection} in $\mathbb{F}_q^n$.

\begin{defn} \label{projdef}
    Let $V$ be a subspace of $\mathbb{F}_q^n$ and $E\subseteq \mathbb{F}_q^n$. The projection of $E$ onto $V$ is defined as
    \begin{equation*}
        \pi_V(E)\coloneqq \{x+V^\perp: x\in \mathbb{F}_q^n, (x+V^\perp)\cap E\neq \varnothing\}.
    \end{equation*}
\end{defn}

We are interested in estimating the cardinality of the \emph{exceptional set}, defined as
\begin{equation*}
    \{V\in G(k,n): |\pi_V(E)|\leq u\}
\end{equation*}
for $u>0$. The case of interest is $u < \min \{q^k, |E|\}$, because otherwise the size of the exceptional set is simply $|G(k,n)|\approx q^{k(n-k)}$. We now have all the ingredients to formulate the main result of this subsection. This result is a finite field analogue of a projection theorem obtained in \cite{anaproj}.

\begin{thm}
\label{main projection theorem}    
Let $p\in [2,+\infty)$, $s\in [0,1]$, and $E\subseteq \mathbb{F}_q^n$ be a nonempty $(p,s)$-Salem set. If $0<u\leq \frac{1}{4} q^{\frac{2k}{p}}$, then 
\begin{equation*}
    |\{V\in G(k,n): |\pi_V(E)|\leq u\}|\lesssim_{p,s} u^{\frac{p}{2}}q^{k(n-k)}|E|^{-ps}.
\end{equation*}
\end{thm}

\begin{proof}
    See Theorem 3.4 in \cite{fraser_rakhmonov_1}.
\end{proof}

The fact that the upper bound in Theorem \ref{main projection theorem} depends on $p$ allows one to optimize it by choosing the best $p$ from the allowed range (i.e., such that $u \leq\frac{1}{4} q^{\frac{2k}{p}}$).  This gives Theorem \ref{main projection theorem}   significant flexibility.  For example, by setting $(p,s) = (2, \frac{1}{2})$ in Theorem \ref{main projection theorem}, which is possible since any set is a $(2, \frac{1}{2})$-Salem set, and following Chen's argument \cite{chen}, we obtain \eqref{marstrandfinite} in full generality. This result generalizes a result of Chen \cite{chen} to the case of non-prime fields and also recovers a recent result by Bright and Gan \cite{bright}. However, the optimal $p$ in  Theorem \ref{main projection theorem} may not be $p=2$ and so we often obtain a strengthening of \eqref{marstrandfinite}, at least in cases where good $L^p$ bounds hold for the Fourier transform of $E$. More precisely, suppose $|E| \approx q^\alpha$ for some $\alpha \in (0,n)$ and $u \lesssim q^\beta$ for some $\beta <\min\{k,\alpha\}$. Then Theorem \ref{main projection theorem} gives an asymptotically stronger estimate than \eqref{marstrandfinite} whenever the right-hand side of \eqref{marstrandfinite} is a positive power of $q$ and $E$ is a $(p,s)$-Salem set for some $2<p<\frac{2k}{\beta}$ with
\[
s>\frac{\beta(\frac{p}{2}-1)+\max\{\alpha,k\}}{p\alpha} =\begin{cases}
\frac{\beta}{2\alpha}(1-\frac{2}{p})+\frac{1}{p}, & \text{if } \alpha \geq k, \\
\frac{\beta}{2\alpha}(1-\frac{2}{p})+\frac{k}{p\alpha}, & \text{if } \alpha < k.
\end{cases}
\]

The case $\alpha \geq k$ is particularly appealing  because  \emph{all} $E\subseteq \mathbb{F}_q^n$ are $(p,\frac{1}{p})$-Salem sets for all $p \in [2,\infty]$. Therefore, any improvement over the trivial bound $s \geq \frac{1}{p}$ yields an improvement over \eqref{marstrandfinite} for sufficiently small $\beta$, provided $k \leq \alpha <k(n-k)$.

\subsection{Fourier restriction}

Suppose we have a nonzero, finite, compactly supported Borel measure $\mu$ on $\mathbb{R}^n$. The famous \emph{restriction problem} asks when it is meaningful to restrict the Fourier transform of a function to the support of $\mu$. Interesting cases include when $\mu$ is the surface measure on the sphere, cone, or paraboloid.  

We focus on the $L^2$ theory, where the influential Stein--Tomas restriction theorem provides estimates in terms of the Fourier decay and scaling properties of $\mu$. The  version we state here is due to Bak--Seeger \cite{BS11}.

\begin{thm}[Stein--Tomas]
    Let $\mu$ be a nonzero, finite, compactly supported Borel measure on $\mathbb{R}^n$, and let $0<\alpha,\beta<n$. Suppose that for all $x \in \mathbb{R}^n$ and all $\delta>0$, 
    \begin{equation*}
        \mu(B(x,\delta)) \lesssim \delta^{\alpha},
    \end{equation*}
    and for all $\xi \in \mathbb{R}^n$,
    \begin{equation*}
        |\widehat{\mu}(\xi)| \lesssim |\xi|^{-\frac{\beta}{2}}.
    \end{equation*}
    Then 
    \begin{equation}
    \label{extension estimate}
        \lVert \widehat{f\mu} \rVert_{L^r(\mathbb{R}^n)} \lesssim_{r,\alpha,\beta} \lVert f \rVert_{L^2(\mu)}
    \end{equation}
    holds for all functions $f \in L^2(\mu)$ and all $r \geq 2 + \frac{4(n-\alpha)}{\beta}$.
\end{thm}

Formally, the estimate \eqref{extension estimate} is an $L^2 \to L^r$ extension estimate; however, by duality, it is equivalent to the $L^{r'} \to L^2$ restriction estimate:  
\begin{equation*}
    \| \widehat f \|_{L^2(\mu)} \lesssim \| f \|_{L^{r'}(\mathbb{R}^n)},
\end{equation*}
where $r'$ is the Hölder conjugate of $r$.

Mockenhaupt and Tao \cite{moctao} proved a finite field analogue of the Stein--Tomas restriction theorem. Analogous to the classical result, their theorem provides a range based on uniform bounds for the Fourier transform of the measure.

Before stating the results, we introduce some notation and definitions. A \emph{probability measure} $\mu$ on $\mathbb{F}_q^n$ is a non-negative function that sums to 1. For $E \subseteq \mathbb{F}_q^n$, the \emph{surface measure} on $E$ is the uniform probability measure, that is, 
\[
\mu(x) \coloneqq \frac{E(x)}{|E|}.
\]

For a function $f: \mathbb{F}_q^n \to \mathbb{C}$, we define
\begin{equation*}
\|f\|_{L^r(\mathbb{F}_q^n)} \coloneqq \Bigg(\sum_{x \in \mathbb{F}_q^n} |f(x)|^r \Bigg)^{\frac{1}{r}}, 
\qquad
\|f\|_{L^r(\mu)} \coloneqq \Bigg(\sum_{x \in \mathbb{F}_q^n} |f(x)|^r \mu(x) \Bigg)^{\frac{1}{r}}.
\end{equation*}

Now we have all the ingredients to state the Mockenhaupt--Tao result (using our notation and terminology).

\begin{thm}[Mockenhaupt--Tao] \label{moctaoresult}
Let $0<\alpha<n$, and let $E \subseteq \mathbb{F}_q^n$ be such that $|E| \approx q^\alpha$. Suppose that $E$ is an $(\infty,s_\infty)$-Salem set. Then, for $\mu$ the surface measure on $E$,
\[
    \|\widehat{f\mu}\|_{L^r(\mathbb{F}_q^n)}\lesssim \|f\|_{L^2(\mu)}
\]
holds for all functions $f:\mathbb{F}_q^n\to \mathbb{C}$, provided that 
\[
r \geq  2+\frac{2(n-\alpha)}{ \alpha s_\infty} .
\]
\end{thm}

In \cite{fraser_rakhmonov_2}, we improved the Mockenhaupt--Tao result using the $L^p$ averages approach.  This result is a finite fields analogue of a Euclidean restriction theorem obtained in \cite{anarestriction}.

\begin{thm}  
\label{cormain}
Let $0<\alpha<n$, and let $E \subseteq \mathbb{F}_q^n$ be such that $|E| \approx q^\alpha$. Suppose that $E$ is a $(p,s)$-Salem set with $s \geq \frac{n}{p\alpha}$. Then, for $\mu$ the surface measure on $E$,
\[
    \|\widehat{f\mu}\|_{L^r(\mathbb{F}_q^n)}\lesssim \|f\|_{L^2(\mu)}
\]
holds for all functions $f:\mathbb{F}_q^n\to \mathbb{C}$, provided that 
\[
r \geq  2+\frac{(2p-2)(n-\alpha)}{ \alpha p s-\alpha} .
\]
In particular, this improves upon the Mockenhaupt--Tao range when
\[
s > s_\infty + \frac{1-s_\infty}{p},
\]
where $s_\infty$ is chosen optimally so that $E$ is an $(\infty, s_\infty)$-Salem set.
\end{thm}

\begin{proof}
See Corollary 2.2 in \cite{fraser_rakhmonov_2}.
\end{proof}

The above restriction theorem has a nice application to Sidon sets. A \emph{Sidon set} $E \subseteq \mathbb{F}_q^n$ is a set in which the equation $a+b=c+d$ implies $\{a,b\} = \{c,d\}$ for every $(a,b,c,d) \in E^4$. As a consequence, if $E$ is Sidon, then $|E| \lesssim q^{\frac{n}{2}}$, but it is easy to construct Sidon sets with $|E| \approx q^{\frac{n}{2}}$. The Sidon sets we consider (i.e., $E$ with $|E| \approx q^{\frac{n}{2}}$) may not exhibit any uniform Fourier decay; see \cite[Proposition 5.2]{fraser_rakhmonov_2} for examples. Therefore, the Mockenhaupt--Tao result alone does not yield a non-trivial range for Fourier restriction.  

\begin{cor} \label{sidonrestriction}
Let $E \subseteq \mathbb{F}_q^n$ be a Sidon set with $|E| \approx q^{\frac{n}{2}}$, and let $\mu$ be the surface measure on $E$. Then 
\[
    \|\widehat{f\mu}\|_{L^8(\mathbb{F}_q^n)}\lesssim \|f\|_{L^2(\mu)}
\]
holds for all functions $f:\mathbb{F}_q^n \to \mathbb{C}$.
\end{cor}

\begin{proof}
See Corollary 5.1 in \cite{fraser_rakhmonov_2}.
\end{proof}

The cardinality assumption on $E$ in the previous result is close to optimal. Indeed, suppose $E \subseteq \mathbb{F}_q^{n-1}$ is a Sidon set with $|E| \approx q^{\frac{n-1}{2}}$, and embed it as a subset of $\mathbb{F}_q^n$. Then, as shown in \cite{fraser_rakhmonov_2}, the restriction estimate fails for \emph{all} $r < \infty$.

We can also apply the above restriction theorem to the Hamming varieties, which were introduced earlier.

\begin{cor} \label{hammingresult}
Let $H_j$ be a Hamming variety in $\mathbb{F}_q^n$, and let $\mu_j$ be the surface measure on $H_j$. Then 
\begin{equation*}
    \lVert \widehat{f\mu_j}\rVert_{L^r(\mathbb{F}_q^n)} \lesssim \lVert f\rVert_{L^2(\mu_j)}
\end{equation*}
holds for all functions $f:\mathbb{F}_q^n \to \mathbb{C}$, provided that 
\[
r \geq \frac{3n-1}{n-1}.
\]
\end{cor}

We note that the Mockenhaupt--Tao result (Theorem~\ref{moctaoresult}) gives a weaker range for $r$, namely $r \geq 4$. In \cite{hamming}, an even better range $r \geq \frac{2(n+1)}{n-1}$ is obtained; however, it is conjectured that the sharp range is in fact $r \geq \frac{2n}{n-1}$.


\bigskip

\begin{thebibliography}{BGGIST07}

\bibitem[BS11]{BS11}J.-G. Bak and A. Seeger. Extensions of the Stein--Tomas Theorem, \emph{Math. Res. Lett.}, \textbf{18}(4),  767--781, (2011).

\bibitem[BKT04]{bourgain}
J. Bourgain, N. Katz and  T. Tao. A sum-product estimate in finite fields, and applications, {\em GAFA}, \textbf{14}, 27-57, (2004). 

\bibitem[BG23+]{bright}
P. Bright and S. Gan.
 Exceptional set estimates in finite fields
 preprint:	\href{https://arxiv.org/abs/2302.13193}{arXiv:2302.13193} (2023). 

\bibitem[CFdO24+]{anarestriction}
M. Carnovale, J. M. Fraser and A. E. de Orellana.
$L^2$ restriction estimates from the Fourier spectrum, 
preprint:	\href{https://arxiv.org/abs/2412.14896}{arXiv:2412.14896} (2024). 



\bibitem[Che18]{chen}
C. Chen. Projections in vector spaces over finite fields, \emph{Ann. Acad. Sci. Fenn. Math.}, {\bf 43},  171-185, (2018).

  \bibitem[CKP22]{hamming}
 D. Cheong, D. Koh and T. Pham.
Extension theorems for Hamming varieties over finite fields, \emph{Proc. Amer. Math. Soc.}, {\bf 150},  161--170, (2022).

\bibitem[Cov20]{covert}
D. J. Covert.
\textit{The finite field distance problem}, MAA Press, vol. 37,  (2020).


 \bibitem[Dvi09]{dvir}
Z. Dvir.  On the size of Kakeya sets in finite fields. \emph{J. Amer. Math. Soc.}, {\bf  22},  (2009),   1093--1097.  

\bibitem[Erd46]{erdos}
P. Erdös. On sets of distances of n points. {\em Amer. Math. Monthly}. \textbf{53} pp. 248-250 (1946).

\bibitem[FFJ15]{projectionss}
K. J. Falconer, J. M. Fraser and X. Jin. Sixty Years of Fractal Projections, 
        Fractal Geometry and Stochastics V, Birkhäuser, Progress in Probability {\bf 70}, (2015). 

 
\bibitem[Fra24+]{fraserfinite}
J. M. Fraser. $L^p$ averages of the Fourier transform in finite fields. preprint:	\href{https://arxiv.org/abs/2407.08589}{arXiv:2407.08589} (2024). 


\bibitem[FdO24+]{anaproj}
  J. M. Fraser and A. E. de Orellana.
A Fourier analytic approach to exceptional set estimates for orthogonal projections, 
\emph{Indiana Univ. Math. J.} (to appear), preprint available at:	\href{https://arxiv.org/abs/2404.11179}{arXiv:2404.11179} (2024). 
 


\bibitem[FR25+]{fraser_rakhmonov_1}
J. M. Fraser and F. Rakhmonov. Exceptional projections in finite fields: Fourier analytic bounds and incidence geometry. preprint:	\href{https://arxiv.org/abs/2503.15072}{arXiv:2503.15072} (2025). 

\bibitem[FR25++]{fraser_rakhmonov_2}
J. M. Fraser and F. Rakhmonov. An improved $L^2$ restriction theorem in finite fields. preprint:	\href{https://arxiv.org/abs/2505.09293}{2505.09293} (2025). 


\bibitem[HIKR11]{hart}
D. Hart, A. Iosevich, D. Koh  and M. Rudnev.
Averages over hyperplanes, sum-product theory in vector spaces over finite fields and the Erd\H{o}s--Falconer distance conjecture, \emph{Trans. Amer. Math. Soc.}, {\bf  363}, (2011),  3255--3275.




\bibitem[IMP11]{imp}
 A. Iosevich, H. Morgan  and J. Pakianathan. On directions determined by subsets
of vector spaces over finite fields, {\it Integers}, {\bf 11}, (2011).

\bibitem[IR07]{iosevich}
A. Iosevich and  M. Rudnev.
Erd{\H{o}}s distance problem in vector spaces over finite fields, 
\emph{Trans. Amer. Math. Soc.}, {\bf  359},  (2007),   6127--6142.

\bibitem[IKR24]{firdavs}
A. Iosevich, D. Koh and  F. Rakhmonov.
The quotient set of the quadratic distance set over finite fields,
{\it Forum Math.}, {\bf 36}, (2024), 1341--1358.


\bibitem[KKR25+]{kang_koh_rakhmonov}
H. Kang, D. Koh and F. Rakhmonov. The Erd{\H{o}}s-Falconer distance problem between arbitrary sets and
$k$-coordinatable sets in finite fields. preprint:	\href{https://arxiv.org/abs/2506.07251}{arXiv:2506.07251} (2025). 

\bibitem[LN97]{lidl}
R. Lidl and H. Niederreiter.
\emph{Finite fields},   Second edition. Encyclopedia of Mathematics and its Applications, {\bf 20}, (Cambridge University Press,  1997).
 

\bibitem[Mat75]{mattila75}P. Mattila. Hausdorff dimension, orthogonal projections and intersections with planes. \emph{Ann. Fenn. Math.}, \textbf{1}(2), 227--244. 

\bibitem[MT04]{moctao}
G. Mockenhaupt and T. Tao. Restriction and Kakeya phenomena for finite fields, {\em 
Duke Math. J.}, {\bf 121}, 35--74, (2004).

  
\bibitem[PS00]{peres}Y. Peres and W. Schlag. Smoothness of projections, Bernoulli convolutions, and the dimension of exceptions. \emph{Duke Math. J.}, \textbf{102}, (2000), 193--251.

\bibitem[Rot53]{roth}
K. Roth. On certain sets of integers. {\em J. London Math. Soc.} \textbf{28}, pp. 104-109, (1953).

\bibitem[TV06]{taovu}
T. Tao and V. Vu.
\emph{Additive combinatorics}, Cambridge Studies in Advanced Mathematics, {\bf 105},  Cambridge University Press,  (2006).

\bibitem[Wol99]{wolff}
 T. Wolff. Recent work connected with the Kakeya problem. \emph{Prospects in mathematics (Princeton, NJ, 1996)}, 129--162, Amer. Math. Soc., Providence, RI, (1999).


\end{thebibliography}
\end{document}